\documentclass[10pt]{amsart}
\usepackage{amsmath,amsfonts, amssymb}
\usepackage[margin=1.5in]{geometry}
\newtheorem{theorem}{Theorem}[section]
\newtheorem{lemma}[theorem]{Lemma}

\newtheorem{corollary}{Corollary}[section]
\theoremstyle{definition}

\theoremstyle{remark}

\numberwithin{equation}{section}

\begin{document}

\title[Entropy production for ES-BGK model]{Entropy production for ellipsoidal BGK model of the Boltzmann equation}


\author{Seok-Bae Yun}
\address{Department of mathematics, Sungkyunkwan University, Suwon 440-746, Republic of Korea }
\email{sbyun01@skku.edu}




\begin{abstract}
The ellipsoidal BGK model (ES-BGK) is a generalized version of the BGK model of the Boltzmann equation designed to yield the correct Prandtl number in the Navier-Stokes
limit. In this paper, we make two observations on the entropy production functional of the ES-BGK model. First, we show that the Cercignani type estimate holds for the ES-BGK model in the whole range of parameter $-1/2<\nu<1$.
Secondly, we observe that the ellipsoidal relaxation operator satisfies an unexpected sign-definite property.
Some implications of these observations are also discussed.
\end{abstract}

\maketitle
\section{Introduction}
The Boltzmann equation is the fundamental model bridging the particle level description and the hydrodynamic description of gases. The application of the Boltzmann equation, however, has been restrictive due mainly to  the excessive cost involved in the numerical computation of the collision integral. In this regard, Bhatnagar, Gross, Krook \cite{BGK} and,
independently, Walender \cite{Wal} suggested a relaxation type model called the BGK model. Since then, this model has been
replacing the Boltzmann equation in various numerical computations, yielding qualitatively satisfactory results
at much lower computational cost compared to that of the Boltzmann equation. There are, however, also several shortcomings reported, with the most notable one being the non-physical Prandtl number - the ratio between the viscosity and the thermal conductivity - it provides. In search of a model equation with
the correct Prandtl number, Holway \cite{Holway} derived a new equation by generalizing the local Maxwellian in the original BGK relaxation operator into
a non-isotropic Gaussian so that the stress tensor can be treated more sophisticatedly:
\begin{eqnarray}\label{ESBGK}
\begin{split}
&\partial_t f+v\cdot \nabla f=A_{\nu}(\mathcal{M}_{\nu}(f)-f),\cr
&\qquad f(x,v,0)=f_0(x,v),
\end{split}
\end{eqnarray}
which is called the ellipsoidal BGK model (ES-BGK model). $f(x,v,t)$ is the number density on the phase space of position and velocity $(x,v)\in \Omega_x\times \mathbb{R}^3$ at time $t\in \mathbb{R}^+$. Throughout this paper, $\Omega_x$
denotes $\mathbb{R}^3$ or $\mathbb{T}^3$.
The non-isotropic Gaussian $\mathcal{M}_{\nu}$ takes the following form:
\begin{eqnarray*}
\mathcal{M}_{\nu}(f)=\frac{\rho}{\sqrt{\det(2\pi\mathcal{T}_{\nu}})}\exp\left(-\frac{1}{2}(v-U)^{\top}\mathcal{T}^{-1}_{\nu}(v-U)\right),
\end{eqnarray*}
where $\rho$, $U$, $T$ and $\Theta$ are macroscopic density, bulk velocity, temperature and stress tensor $\Theta$ respectively:
\begin{eqnarray*}
\rho(x,t)&=&\int_{\mathbb{R}^3_v}f(x,v,t)dv,\cr
\rho(x,t)U(x,t)&=&\int_{\mathbb{R}^3_v}f(x,v,t)vdv,\cr
3\rho(x,t) T(x,t)&=&\int_{\mathbb{R}^3_v}f(x,v,t)|v-U(x,t)|^2dv,\cr
\rho(x,t)\Theta(x,t)&=&\int_{\mathbb{R}^3_v}f(x,v,t)(v-U)\otimes(v-U)dv,
\end{eqnarray*}
and $\mathcal{T}_{\nu}$ denotes the temperature tensor defined by
\begin{eqnarray*}
\mathcal{T}_{\nu}&=&\left(
\begin{array}{ccc}
(1-\nu) T+\nu\Theta_{11}&\nu\Theta_{12}&\nu\Theta_{13}\cr
\nu\Theta_{21}&(1-\nu)T+\nu\Theta_{22}&\nu\Theta_{23}\cr
\nu\Theta_{31}&\nu\Theta_{32}&(1-\nu) T+\nu\Theta_{33}
\end{array}
\right)\cr
&=&(1-\nu) T Id+\nu\Theta.
\end{eqnarray*}
The relaxation operator satisfies the following cancellation property:
\begin{eqnarray*}
\int_{\mathbb{R}^3_v} \big(\mathcal{M}_{\nu}(f)-f\big)\left(\begin{array}{c}1\cr v\cr|v|^2\end{array}\right)dv=0,
\end{eqnarray*}
leading to the conservations of mass, momentum and energy. H-theorem was not verified when this model was first suggested, and
proved only recently by Andries et al \cite{ALPP} (See also \cite{Brull1,Brull2}):
\[
\frac{d}{dt}\int_{\mathbb{R}^3_v}f\ln fdv\leq 0.
\]
The collision frequency $A_{\nu}$ takes the following form:
\begin{eqnarray}\label{A_nu}
A_{\nu}=\sigma(\rho,T)/(1-\nu),\quad -1/2<\nu< 1,
\end{eqnarray}
 where  $\sigma(x,y)=x^{\alpha}y^{\beta}$ for some $\alpha$, $\beta\geq0$. An application of the Chapmann-Enskog expansion shows that the Prandtl number is given by $Pr=1/(1-\nu)$ (See \cite{ALPP,CC,Holway,Stru-book}). Therefore, we can obtain the correct Prandtl number by adjusting the free parameter $\nu$. Two important cases are $\nu=(Pr-1)/Pr\approx-1/2$ and $\nu=0$.
The former gives the correct Prandtl number, while the latter corresponds to the original BGK model.\newline

The result of this paper is two-fold. First, we establish the following entropy-entropy production estimate:
\begin{eqnarray*}
D_{\nu}(f)\geq C_{\nu}H(f|\mathcal{M}_0),
\end{eqnarray*}
where $\mathcal{M}_0$ denotes the local Maxwellian constructed from $f$, and $C_{\nu}=\min\{1+2\nu,1-\nu\}A_{\nu}$.
Note that the corresponding estimate for the Boltzmann equation takes the following form \cite{TV,V1}:
\[
D_{BE}(f)\geq C_{\varepsilon}H(f|\mathcal{M}_0)^{1+\varepsilon},
\]
where $\varepsilon$ in the exponent, although it can be taken to be arbitrarily small with $C_{\epsilon}$ adjusted accordingly, can never be removed \cite{BC,W}.
Therefore, our result indicates that the ellipsoidal modification of the classical BGK model is still not sophisticate enough to capture the subtle entropy production mechanism of the Boltzmann equation: In terms of entropy production,
the BGK model behaves like the linear Boltzmann equation \cite{BCL}.
The proof is by straightforward computations making combined use of the convexity of $x\ln x$ and the concavity of $\ln x$. The following estimate comparing the entropy difference between various Maxwellians and Gaussians are crucially used:
\begin{eqnarray*}
H\big(\mathcal{M}_0\big)-H\big(\mathcal{M}_{\nu}\big)\geq \max\{\nu,-2\nu\}\big\{H\big(\mathcal{M}_0\big)-H\big(\mathcal{M}_1\big)\big\}.
\end{eqnarray*}

Secondly, we observe that the relaxation operator of the ellipsoidal BGK model has an interesting sign-definite property. We recall that, in the case of the Boltzmann equation or the original BGK model, the entropy dissipation is a
direct consequence of the following elementary inequality:
\begin{eqnarray}\label{elementary}
\mathcal{E}(a,b)\equiv(a-b)(\ln a-\ln b)\geq0.
\end{eqnarray}
The entropy production functional for the Boltzmann equation $D_{BE}(f)$ and the original BGK model $D_{0}(f)$ can be written respectively in the following manner:
\begin{align*}
D_{BE}(f)&=\int_{\mathbb{R}^3_v\times \mathbb{R}^3_{v_*}}B(v,v_*,\omega)\mathcal{E}\big(f^{\prime}f^{\prime}_*, ff_*\big)d\omega dvdv_*,\cr
D_0(f)&=\int_{\mathbb{R}^3_v}A_0\mathcal{E}\big(\mathcal{M}_0(f), f\big)dv,
\end{align*}
where $B$ denotes the Boltzmann collision kernel \cite{C,CIP,V}.
However, we immediately see that this is not the case for the ES-BGK model:
An explicit computation gives rise to an additional term $R_{\nu}(f)$ other than the one characterized by (\ref{elementary}):
\begin{align}\label{add}
D_{\nu}(f)&\equiv-\int_{\mathbb{R}^3_v}A_{\nu}\{\mathcal{M}_{\nu}(f)- f\}\ln fdv\cr
&=\int_{\mathbb{R}^3_v}A_{\nu}\mathcal{E}\big(\mathcal{M}_{\nu}(f), f\big)dv+R_{\nu}(f)
\end{align}
where the remainder term is given by
\begin{eqnarray}\label{R}
R_{\nu}( f)=\Big\{\int_{\mathbb{R}^3_v}A_{\nu}\big\{\mathcal{M}_{\nu}(f)-f\big\}(v-U)\otimes(v-U)dv\Big\}:\mathcal{T}^{-1}_{\nu}.
\end{eqnarray}
This shows that the inequality (\ref{elementary}) alone is, unlike the Boltmzann equation or the original BGK model, not sufficient to
determine the non-negativity of the entropy production functional $D_{\nu}(f)$.
The remainder term arises due to the non-conservative terms inside the ellipsoidal Gaussian, namely the stress tensor.
A rather unexpected property of this remainder $R_{\nu}$ is that it is sign definite with respect to $\nu$
in the following sense:
\[
\left\{\int_{\mathbb{R}^3_v}(\mathcal{M}_{\nu}(f)-f)(v-U)\otimes(v-U)dv\right\}:\mathcal{T}^{-1}_{\nu}\geq (\leq)~ 0\quad \mbox{ if }~\nu\geq (\leq)~0.
\]
This remainder term vanishes when $\nu=0$, which corresponds to the original BGK  model.\newline

A brief referential check is in order.
The first existence result for the BGK model goes back to  \cite{Perthame}, where the Cauchy problem were considered in the framework of
weak solutions. The uniqueness problem in a weighted $L^{\infty}$ setting  was studied in \cite{PP}.
These theories were then generalized to various directions such as the whole space case \cite{Mischler}, $L^p$ setting \cite{ZH}, and the
BGK models with external forces \cite{BoC,WZ} or mean field effect \cite{Zhang}.
Ukai studied the stationary problem in a bounded interval in \cite{Ukai-BGK}. For the results on the existence of classical solutions
and their asymptotic stability near global Maxwellians, see \cite{Bello,Chan,Yun}.
Works on various macrosopic limits such as the diffusion limit and the hydrodynamic limit can be found in
\cite{BV,DMOS,MMM,SR1,SR2}.
Literatures on the numerical computations of BGK type models are abundant, which is natural considering that it was the reason why the BGK model was first suggested.
We do not attempt to present complete list of them. Interested readers may refer to \cite{ABLP,ALPP,AKT,FJ,FR,GT,LPS,M,MS,PPuppo} and references therein.
For rigorous convergence analysis of numerical schemes, see \cite{Issautier,RSY}.

The recent revival of interest on the ES-BGK model can large be attributed to the establishment of the $H$-theorem accomplished in \cite{ALPP}.
Alternate proof was presented later by Brull et al. \cite{Brull1,Brull2}.  The mathematical study of the ES-BGK model such as the existence theory is in its initial stage. The existence of classical solutions was studied when the solution lies close to a global Maxwellian \cite{Yun3}, or when the collision frequency does not depend on macroscopic fields \cite{Yun2}.
For works on the existence of weak solutions, see \cite{PY}. Nice survey on mathematical and physical theory of kinetic equations
can be found in \cite{Bird,C,CIP,CC,GL,Sone,Sone2,Stru-book,UT,V}.
\newline

We define some notations used throughout this paper.
\begin{itemize}
\item If not stated otherwise, constants will be defined generically
\item $C_{a,b,c,..}$ denotes a generic constant which depends, not necessarily exclusively, on $a,b,c...$.
\item For $\kappa\in\mathbb{R}^3_v$, $\kappa^{\top}$ denotes its transpose.
\item Id denotes the $3\times 3$ identity matrix.
\item The Frobenius product $A:B$ denotes
\[
A:B=\sum_{1\leq i,j\leq 3}A_{ij}B_{ij}.
\]
\item We use $T$ to denote the local temperature, while $T^f$ denotes the final time.\newline
\end{itemize}

This paper is organized as follows. In section 2, we prove that the entropy production functional satisfies the Cercignani type inequality.
Exponentially fast convergence to equilibrium in the homogeneous case is derived.
In section 3, the positive definite property of the ellipsoidal relaxation operator is considered. It is shown that this property
can be used to derive the compactness in $L^1$ of the ellipsoidal Gaussian.
\newline\newline
%
%
%
\section{Entropy-Entropy production estimate for ES-BGK model}
We first define the $H$-functional $H(f)$, the relative entropy $H(f|g)$ and the entropy production functional $D_{\nu}(f)$:
\begin{eqnarray*}
H(f)&=&\int_{\mathbb{R}^3_v}f\ln fdv,\cr
H(f|g)&=&\int_{\mathbb{R}^3_v}f\ln \big(f/g\big)dv,\cr
D_{\nu}(f)&=&-\int_{\mathbb{R}^3_v}A_{\nu}\big(\mathcal{M}_{\nu}(f)-f\big)\ln fdv.
\end{eqnarray*}
Recall from (\ref{add}) that, unlike the original BGK model,
\[
D_{\nu}(f)\neq \int_{\mathbb{R}^3_v}A_{\nu}\big(\mathcal{M}_{\nu}(f)-f\big)\big(\ln \mathcal{M}_{\nu}(f)-\ln f\big)dv,
\]
except for the case $\nu=0$.
We also observe that the cases $\nu=0$ and $\nu=1$ correspond respectively to the local Maxwellian and the multivariate Gaussian with the stress tensor $\Theta$ as its covariance:
\begin{eqnarray*}
\mathcal{M}_{0}(f)&=&\frac{\rho}{\sqrt{(2\pi T)^3}}\exp\left(-\frac{|v-U|^2}{2T}\right),\cr
\mathcal{M}_{1}(f)&=&\frac{\rho}{\sqrt{\det(2\pi \Theta)}}\exp\left(-\frac{1}{2}(v-U)^{\top}\Theta^{-1}(v-U)\right).
\end{eqnarray*}
When there's no risk of confusion, we suppress the dependence on $f$ and write $\mathcal{M}_{\nu}$ instead of $\mathcal{M}_{\nu}(f)$ for simplicity.

\begin{theorem}\label{main 1} The entropy production functional $D_{\nu}$ of the ES-BGK model satisfies
\begin{eqnarray*}
D_{\nu}(f)\geq \min\{1+2\nu,1-\nu\}A_{\nu}
H(f|\mathcal{M}_{0})
\end{eqnarray*}
for $-1/2<\nu<1$.
\end{theorem}
We start with following lemma, which is crucially used in the proof of the theorem.

%
%
%
%
\begin{lemma}\label{main lemma}For $-1/2<\nu<1$
\begin{eqnarray*}
H\big(\mathcal{M}_0\big)-H\big(\mathcal{M}_{\nu}\big)\geq \max\{\nu,-2\nu\}\big\{H\big(\mathcal{M}_0\big)-H\big(\mathcal{M}_1\big)\big\}.
\end{eqnarray*}
\end{lemma}
\begin{proof}
An explicit computation gives
\begin{eqnarray*}
H(\mathcal{M}_0)&=&\int_{\mathbb{R}^3_v}\mathcal{M}_0\ln \mathcal{M}_0dv\cr
&=&\int_{\mathbb{R}^3_v}\mathcal{M}_0\Big\{\ln\frac{\rho}{(2\pi T)^{3/2}}-\frac{|v-U|^2}{2T}\Big\}dv\cr
&=&\rho\ln\frac{\rho}{(2\pi T)^{3/2}}-\frac{3}{2}\rho,
\end{eqnarray*}
and
\begin{eqnarray*}
H(\mathcal{M}_{\nu})&=&\int_{\mathbb{R}^3_v}\mathcal{M}_{\nu}\ln \mathcal{M}_{\nu} dv\cr
&=&\int_{\mathbb{R}^3_v}\mathcal{M}_{\nu}\Big\{\ln\frac{\rho}{\{\det(2\pi \mathcal{T}_{\nu})\}^{1/2}}-\frac{1}{2}(v-U)^{\top}\mathcal{T}_{\nu}^{-1}(v-U)\Big\}dv\cr
&=&\rho\ln\frac{\rho}{\{\det(2\pi \mathcal{T}_{\nu})\}^{1/2}}-\frac{1}{2}\rho\mathcal{T}_{\nu}:\mathcal{T}_{\nu}^{-1}.
\end{eqnarray*}
We then recall the identity
\begin{eqnarray*}
A:B=tr(A^{\top}B)
\end{eqnarray*}
and use the the symmetry of $\mathcal{T}_{\nu}$ to compute
\begin{eqnarray*}
\mathcal{T}_{\nu}:\mathcal{T}_{\nu}^{-1}=tr\big(\mathcal{T}^{\top}_{\nu}\mathcal{T}_{\nu}^{-1}\big)
=tr\big(\mathcal{T}_{\nu}\mathcal{T}_{\nu}^{-1}\big)=tr\big(Id\big)=3,
\end{eqnarray*}
so that
\begin{eqnarray*}
H(\mathcal{M}_{\nu}(f))=\rho\ln\frac{\rho}{\{\det(2\pi \mathcal{T}_{\nu})\}^{1/2}}-\frac{3}{2}\rho.
\end{eqnarray*}
Therefore,
\begin{eqnarray}\label{rewritten}
H(\mathcal{M}_0)-H(\mathcal{M}_{\nu}(f))=\frac{1}{2}\rho\ln\frac{\det\mathcal{T}_{\nu}}{T^3}.
\end{eqnarray}
We note that, due to the symmetry of $\Theta$, there exists an orthogonal matrix $P$ such that $P^{\top}\Theta P$ is a diagonal matrix.
We denote the diagonal elements as $\theta_i~(i=1,2,3)$, which are non-negative owing to
\[
\kappa^{\top}\Theta\kappa=\frac{1}{\rho}\int_{\mathbb{R}^3_v}f\{(v-U)\cdot\kappa\}^2dv\geq 0,\quad \kappa\in \mathbb{R}^3_v,
\]
to write
\begin{eqnarray*}
P^{\top}\Theta P=\left(\begin{array}{ccc}\theta_1&0&0\cr0&\theta_2&0\cr0&0&\theta_3\end{array}\right).
\end{eqnarray*}
This gives
\begin{eqnarray*}
P^{\top}\mathcal{T}_{\nu}P&=&P^{\top}\{(1-\nu)TId+\nu\Theta\}P\cr
&=&(1-\nu)TId+\nu P^{\top}\Theta P\cr
&=&\left(\begin{array}{ccc}(1-\nu)T+\nu\theta_1&0&0\\0&(1-\nu)T+\nu\theta_2&0\\0&0&(1-\nu)T+\nu\theta_3\end{array}\right).
\end{eqnarray*}
Since the determinant is invariant under similarity transform,
\begin{eqnarray*}
\det\mathcal{T}_{\nu}=\prod_{1\leq i\leq 3}\big\{(1-\nu)T+\nu\theta_i\big\},
\end{eqnarray*}
and (\ref{rewritten}) can be rewritten as
\begin{eqnarray}\label{rewritten2}
 H(\mathcal{M}_0)-H(\mathcal{M}_{\nu})=\frac{1}{2}\rho\sum_{i=1}^3\ln\frac{\big\{(1-\nu)T+\nu\theta_i\big\}}{T}.
\end{eqnarray}
We divide the remaining argument into the following two cases.\newline\newline
(1) $0\leq\nu<1$: From the concavity of $\ln$, we have
\begin{eqnarray*}
\ln \{(1-\nu)T+\nu\theta_i\}\geq (1-\nu)\ln T+\nu \ln \theta_i.
\end{eqnarray*}
Therefore,
\begin{eqnarray*}
 H(\mathcal{M}_0)-H(\mathcal{M}_{\nu})
&=&\frac{1}{2}\rho\left\{\sum_{i=1}^3\ln\big\{(1-\nu)T+\nu\theta_i\big\}-3\ln T\right\}\cr
&\geq&\frac{1}{2}\rho\left\{\sum_{i=1}^3\big\{(1-\nu)\ln T+\nu\ln\theta_i\big\}-3\ln T\right\}\cr
&=&\frac{1}{2}\rho\left\{\nu\sum_{i=1}^3\ln\theta_i-3\nu\ln T\right\}\cr
&=&\frac{\nu}{2}\rho\ln\frac{\theta_1\theta_2\theta_3}{T^3}\cr
&=&\nu\rho\ln\sqrt{\frac{\det \Theta}{T^{3}}}.
\end{eqnarray*}
But an explicit computation gives
\begin{eqnarray*}
\rho\ln\sqrt{\frac{\det \Theta}{T^{3}}}=H(\mathcal{M}_0)-H(\mathcal{M}_1),
\end{eqnarray*}
which yields the desired result.\newline\newline
(2) $-1/2<\nu<0$: 
In this case, $(1-\nu)T+\nu\theta_i$ is not a convex combination of $T$ and $\theta$ anymore. Instead, we observe from the similarity-invariance of the trace operator that
\begin{eqnarray*}
3T=tr(\Theta)=tr(P^{\top}\Theta P)=\theta_1+\theta_2+\theta_3
\end{eqnarray*}
to rewrite $(1-\nu)T+\nu\theta_i$ as
\begin{eqnarray*}
(1-\nu)T+\nu\theta_i&=&\frac{(1-\nu)}{3}(\theta_1+\theta_2+\theta_3)+\nu\theta_i\cr
&=&\frac{1+2\nu}{3}(\theta_1+\theta_2+\theta_3)-\nu\sum_{j\neq i}\theta_j\cr
&=&(1+2\nu)T-\nu\sum_{j\neq i}\theta_j,
\end{eqnarray*}
which is a convex combination of $T$ and $\theta_{j\neq i}$ (j=1,2,3). Now we are able to use the concavity of $ln$:
\begin{eqnarray*}
\ln\left\{(1-\nu)T+\nu\theta_i\right\}\geq (1+2\nu)\ln T-\nu\sum_{j\neq i}\ln\theta_j
\end{eqnarray*}
to proceed similarly as in the previous case:
\begin{eqnarray*}
 H(\mathcal{M}_0)-H(\mathcal{M}_{\nu})
&=&\frac{1}{2}\rho\left\{\sum_{i=1}^3\ln\big\{(1-\nu)T+\nu\theta_i\big\}-3\ln T\right\}\cr
&\geq&\frac{1}{2}\rho\left\{\sum_{i=1}^3\Big\{(1+2\nu)\ln T-\nu\sum_{j\neq i}\ln\theta_j\Big\}-3\ln T\right\}\cr
&=&-\frac{1}{2}\rho\left\{2\nu\sum_{i=1}^3\ln\theta_i-6\nu\ln T\right\}\cr
&=&-2\nu\rho\ln\sqrt{\frac{\theta_1\theta_2\theta_3}{T^3}}\cr
&=&-2\nu\rho\ln\sqrt{\frac{\det\Theta}{T^{3}}}\cr
&=&-2\nu\left\{H(M_0)-H(M_1)\right\}.
\end{eqnarray*}
Combining the case (1) and (2) gives the desired result.
\end{proof}
The following lemma says that the non-multivariate Gaussian also forms an important class of distribution functions with ``small entropy".
The proof can be found, for example, in \cite{ALPP}. We present the proof for the readers' convenience.
\begin{lemma}\cite{ALPP}\label{main lemma 2} The H-functional of $\mathcal{M}_{1}(f)$ and $\mathcal{M}_{0}(f)$ satisfies
\[
H(\mathcal{M}_{0})\leq H(\mathcal{M}_{1})\leq H(f).
\]
\end{lemma}
\begin{proof}
The first inequality is well-known. For the second one, we use the convexity of $x\ln x$, and
the fact that $\mathcal{M}_{1}(f)$ and $f$ share the same mass and the same stress tensor to compute
\begin{align*}
H(f)&\geq H(\mathcal{M}_{1})+\int_{\mathbb{R}^3_v}H^{\prime}(\mathcal{M}_{1})(f-\mathcal{M}_{1})dv\cr
&\geq H(\mathcal{M}_{1})+\int_{\mathbb{R}^3_v}\left\{1+\ln\frac{\rho}{\sqrt{\det(2\pi\Theta)}}-\frac{1}{2}(v-U)^{\top}\Theta^{-1}(v-U)\right\}\big\{f-\mathcal{M}_{1}\big\}dv\cr
&= H(\mathcal{M}_{1}).
\end{align*}
\end{proof}
\subsection{The proof of Theorem 2.1}
Due to the convexity of $x\ln x$, $D_{\nu}(f)$ satisfies
\begin{eqnarray*}
D_{\nu}(f)=-\int_{\mathbb{R}^3_v}A_{\nu}H^{\prime}(f)\big(\mathcal{M}_{\nu}(f)-f\big)dv\geq A_{\nu}\big\{H(f)-H(\mathcal{M}_{\nu})\big\}.
\end{eqnarray*}
We then split $H(f)-H\big(\mathcal{M}_{\nu}(f)\big)$ as
\begin{eqnarray*}
H(f)-H\big(\mathcal{M}_{\nu}(f)\big)&=&H(f)-H(\mathcal{M}_0)+H(\mathcal{M}_0)-H(\mathcal{M}_{\nu})\cr
&=&H(f|\mathcal{M}_0)+\left\{H(\mathcal{M}_0)-H(\mathcal{M}_{\nu})\right\},
\end{eqnarray*}
and apply Lemma \ref{main lemma} and Lemma \ref{main lemma 2} to find
\begin{eqnarray*}
H(f)-H\big(\mathcal{M}_{\nu}\big)&\geq& H(f|\mathcal{M}_0)+\max\{\nu,-2\nu\}\left\{H(\mathcal{M}_0)-H(\mathcal{M}_{1})\right\}\cr
&\geq& H(f|\mathcal{M}_0)+\max\{\nu,-2\nu\}\left\{H(\mathcal{M}_0)-H(f)\right\}\cr
&=& H(f|\mathcal{M}_0)-\max\{\nu,-2\nu\}H(f|\mathcal{M}_0)\cr
&=&\min\{1-\nu,1+2\nu\}H(f|\mathcal{M}_0).
\end{eqnarray*}
This completes the proof.

A direct consequence of Theorem 2.1, of course is the H-theorem:
\begin{corollary} H-functional for the ellipsoidal BGK model (\ref{ESBGK}) is non-increasing in time.
\end{corollary}
\begin{proof}
Multiplying $\ln f$ to (\ref{ESBGK}) and Integrating in time, we get
\begin{eqnarray}\label{corol1}
H(f(t))+\int^t_0D_{\nu}(f(s))ds=H(f_0).
\end{eqnarray}
Since $H(f|\mathcal{M}_0)\geq0$, we have from Theorem 2.1 that
\begin{eqnarray*}
D_{\nu}(f)\geq \min\{1-\nu,1+2\nu\}A_{\nu}H(f|\mathcal{M}_0)\geq0
\end{eqnarray*}
in the range $-1/2<\nu<1$, which completes the proof.
\end{proof}
\subsection{Equilibrium states are Maxwellians, not ES-Gaussians}
The relaxation operator for the ES-BGK model leaves the possibility that the non-isotropic Gaussian, not the local Maxwellian, can be an equilibrium state.
This is not a good news since it implies that the ES-BGK model may not correctly capture the asymptotic behavior of the Boltzmann equation.
Theorem 2.1, however, leads us to the conclusion that the only possible equilibrium solution is, as for the Boltzmann equation or the original BGK model, is the local Maxwellian. To show this, set
\[
\mathcal{M}_{\nu}(f)-f=0.
\]
Then, from Theorem 2.1,
\[
0=-\int_{\mathbb{R}^3_v}A_{\nu}(\mathcal{M}_{\nu}(f)-f)\ln fdv\equiv D_{\nu}(f)\geq H(f|\mathcal{M}_{0})\geq0,
\]
which implies $H(f|\mathcal{M}_{0})=0$, or
\[
f=\mathcal{M}_0=\frac{\rho}{(2\pi T)^{3/2}}\exp\left(-\frac{|v-U|^2}{2T}\right).
\]
\subsection{Exponentially fast stabilization to equilibrium in the homogeneous case}
Theorem \ref{main 1}, from a standard argument, leads to the exponential convergence to the equilibrium in the spatially homogeneous case.
Note that this is not the case for the  Boltzmann equation in general, even in the spatially homogeneous setting.
\begin{theorem} For spatially homogeneous ES-BGK model, we have
\begin{eqnarray*}
\|f(t)-\mathcal{M}_0\|_{L^1_v}\leq e^{-\frac{3}{2}\min\left\{1,\frac{1+2\nu}{1-\nu}\right\}t}
\sqrt{2H(f_0|\mathcal{M}_0)}.
\end{eqnarray*}
for $-1/2<\nu<1$.
\end{theorem}
We multiply $\ln (f/\mathcal{M}_0)$ on both sides of ES-BGK model, integrate in $v$ and apply Theorem 2.1 to get
\begin{eqnarray*}
\frac{d}{dt}H(f|\mathcal{M}_0)&\leq&-\min\{1+2\nu,1-\nu\}A_{\nu}H(f|\mathcal{M}_0)\cr
&=&-3\min\left\{1,\frac{1+2\nu}{1-\nu}\right\}H(f|\mathcal{M}_0),
\end{eqnarray*}
where we used the fact that $A_{\nu}=3/(1-\nu)$ in the homogeneous case. Therefore, by Gronwall's lemma,
\begin{eqnarray*}
H(f|\mathcal{M}_0)\leq e^{-3\min\left\{1,\frac{1+2\nu}{1-\nu}\right\}t}H(f_0|\mathcal{M}_0).
\end{eqnarray*}
We then make use of the Kullback inequality:
\begin{eqnarray*}
\|f-g\|_{L^1}\leq \sqrt{2H(f|g)}\quad \mbox{ if }\quad \int f=\int g
\end{eqnarray*}
to obtain
\begin{eqnarray*}
\|f(t)-\mathcal{M}_0\|_{L^1_v}\leq e^{-\frac{3}{2}\min\left\{1,\frac{1+2\nu}{1-\nu}\right\}t}\sqrt{2H(f_0|\mathcal{M}_0)}.
\end{eqnarray*}
%
%
%
%

%
%
%
\section{Sign-definite property of the ellipsodial relaxation operator}
In this section, we show that the relaxation operator of the ES-BGK model satisfies the following unexpected property:
\begin{theorem}\label{main 2} The remainder functional $R_{\nu}(t)$ defined in (\ref{R})
 satisfies the following sign-definiteness:
\begin{eqnarray*}
\nu R_{\nu}(x,t)\geq 0
\end{eqnarray*}
in the range $-1/2<\nu<1$. Note that $R_0(x,t)\equiv0$.
\end{theorem}
%
%
%
%
\begin{proof}
We recall the definition of $\mathcal{T}_{\nu}$ and $\Theta$ to get
\begin{eqnarray*}
\int_{\mathbb{R}^3_v}\left\{\mathcal{M}_{\nu}(f)-f\right\}(v-U)\otimes(v-U)dv
=\rho\mathcal{T}_{\nu}-\rho\Theta,
\end{eqnarray*}
and evoke the following identity again:
\begin{eqnarray*}
A:B=tr(A^{\top}B),
\end{eqnarray*}
to compute
\begin{align*}
R_{\nu}(t)&=\left\{\int_{\mathbb{R}^3_v}A_{\nu}\left\{\mathcal{M}_{\nu}(f)-f\right\}(v-U)\otimes(v-U)dv\right\}:\mathcal{T}^{-1}_{\nu}\cr
&=A_{\nu}\rho\big\{\mathcal{T}_{\nu}-\Theta\big\}:\mathcal{T}^{-1}_{\nu}\cr
&=A_{\nu}\rho\left\{tr\big(\mathcal{T}^{\top}_{\nu}\mathcal{T}^{-1}_{\nu}\big)-tr\big(\Theta^{\top}\mathcal{T}^{-1}_{\nu}\big)\right\}.
\end{align*}
Since $\mathcal{T}_{\nu}$ and $\Theta$ are symmetric matrices, $R_{\nu}(t)$ can be further simplified as
\begin{eqnarray}\label{view}
R_{\nu}(t)=A_{\nu}\rho\left\{3-tr\big(\Theta\mathcal{T}^{-1}_{\nu}\big)\right\}
\end{eqnarray}
We then recall from the proof of the Theorem \ref{main 1} that
\begin{align*}
P^{\top}\mathcal{T}^{-1}_{\nu} P&=\big\{P^{\top}\mathcal{T}_{\nu} P\big\}^{-1}\cr
&=\left(\begin{array}{ccc}
1/\{(1-\nu)T+\nu\theta_1\}&0&0\\0&1/\{(1-\nu)T+\nu\theta_2\}&0\\0&0&1/\{(1-\nu)T+\nu\theta_3\}
\end{array}\right),
\end{align*}
and use the fact that (1) $\Theta$ and $\mathcal{T}_{\nu}$ are simultaneously diagonaliziable, and (2) the trace operator is similarity invariant, to obtain the following expression for $tr\big(\Theta\mathcal{T}^{-1}_{\nu}\big)$:
\begin{align*}
tr\big(\Theta\mathcal{T}^{-1}_{\nu}\big)&=tr\left[\left(\begin{array}{ccc}\theta_1&0&0\\0&\theta_2&0\\0&0&\theta_3\end{array}\right)
\left(\begin{array}{ccc}\frac{1}{(1-\nu)T+\nu\theta_1}&0&0\\0&1\frac{1}{(1-\nu)T+\nu\theta_2}&0\\0&0&\frac{1}{(1-\nu)T+\nu\theta_3}\end{array}\right)\right]\cr
&=tr\left(\begin{array}{ccc}\frac{\theta_1}{(1-\nu)T+\nu\theta_1}&0&0\\0&1\frac{\theta_2}{(1-\nu)T+\nu\theta_2}&0\\0&0&\frac{\theta_3}{(1-\nu)T+\nu\theta_3}\end{array}\right)\cr
&=\frac{\theta_1}{(1-\nu)T+\nu\theta_1}+\frac{\theta_2}{(1-\nu)T+\nu\theta_2}+\frac{\theta_3}{(1-\nu)T+\nu\theta_3}\cr
&\equiv F_{\nu}(x,t).
\end{align*}
In view of this and (\ref{view}),
the desired estimate now follows directly once the following lemma is established.
\end{proof}
\begin{lemma}\label{F<3} $F_{\nu}$ satisfies
\begin{enumerate}
\item $F_{\nu}=3$ for $\nu=0$,
\item $F_{\nu}\leq 3$ for $0\leq \nu<1$,
\item $F_{\nu}\geq 3$ for $-1/2< \nu\leq0$.
\end{enumerate}
\end{lemma}

\begin{proof}
$(1)$ $\nu=0$: This case is simple since $F$ reduces to
\[\frac{\theta_1}{T}+\frac{\theta_2}{T}+\frac{\theta_3}{T}=3.\]
$(2)$ $0<\nu<1$: We set
\begin{eqnarray*}
A=(1-\nu)T+\nu\theta_1,\quad B=(1-\nu)T+\nu\theta_2,\quad C=(1-\nu)T+\nu\theta_3
\end{eqnarray*}
to see that
\begin{align}\label{twice0}
F_{\nu}&=\frac{\frac{1}{\nu}(A-(1-\nu)T)}{A}+\frac{\frac{1}{\nu}(B-(1-\nu)T)}{B}+\frac{\frac{1}{\nu}(C-(1-\nu)T)}{C}\cr
&=\frac{1}{\nu}\left\{3-(1-\nu)T\left(\frac{1}{A}+\frac{1}{B}+\frac{1}{C}\right)\right\}
\end{align}
Applying the arithmetic inequality twice, we obtain
\begin{eqnarray}\label{twice}
\frac{1}{A}+\frac{1}{B}+\frac{1}{C}\geq\frac{3}{\sqrt[3]{ABC}}\geq \frac{3}{\frac{A+B+C}{3}}=\frac{3}{T}.
\end{eqnarray}
In the last part, we used
\begin{eqnarray*}
\sum\theta_i=tr(P^{\top}\Theta P)=tr(\Theta)=3T,
\end{eqnarray*}
and
\begin{eqnarray*}
A+B+C
=3(1-\nu)T+\nu(\theta_1+\theta_2+\theta_3)=3T.
\end{eqnarray*}
Therefore, from (\ref{twice0}) and (\ref{twice}), we get the desired result:
\begin{eqnarray*}
F_{\nu}\leq\frac{3}{\nu}-\frac{(1-\nu)T}{\nu}\left(\frac{3}{T}\right)\leq 3.
\end{eqnarray*}
$(3)$ $-1/2<\nu<0$: This case follows in the exactly same manner as in $(2)$. The only difference is the
change of direction of the inequality, due to the negative sign of $\nu$.
\end{proof}

%
%
%
%
%
\subsection{Sign-definite property at linearized level} The sign-definite property of the remainder term in entropy production functional can be understood in a more explicit way in the linearized setting.
Let $f=m+\sqrt{m}g$ for $m=1/(2\pi)^{3/2}e^{-|v|^2/2}$, then $(\ref{ESBGK})$ is rewritten as
\[
\partial_tg+v\cdot\nabla_xg=L_{\nu}g+\Gamma(g).
\]
The linearized relaxation operator $L_{\nu}$  takes the following form:
\[
L_{\nu}g=\frac{1}{1-\nu}\left\{(P_0f-f)+\nu (P_1f+P_2f)\right\}.
\]
For the definition of $\Gamma(g)$, which is not relevant for our purpose, see \cite{Yun2}.
$P_0$, $P_1$ and $P_2$ denote the projection operators on the linear spaces spanned respectively by
\[
\{1,v,|v|^2\}, \quad\{3v_i^2-|v|^2\}_{i=1,2,3},~\mbox{ and }~\{v_iv_j\}_{i<j}.
\]
It was shown in \cite{Yun2} that  $P_0\perp(P_1+P_2)$, which gives
\begin{eqnarray*}
-\langle L_{\nu}f,f\rangle_{L^2_v}=\frac{1}{1-\nu}\|(I-P_0)f\|^2_{L^2_v}+\frac{\nu}{1-\nu}\|(P_1+P_2)f\|^2_{L^2_v}.
\end{eqnarray*}
This shows that the linearized entropy dissipation $-\langle L_{\nu}f,f\rangle_{L^2_v}$ is decomposed into the usual dissipation term
and the additional remainder not present in the original BGK model or the Boltzmann equation.
The remainder term $\nu(1-\nu)^{-1}\|(P_1+P_2)f\|^2_{L^2_v}$ clearly is sign definite with respect to $\nu$ in the range $-1/2<\nu<1$.
%
%
%
%
%
\subsection{Weak compactness of $\mathcal{M}_{\nu}(f)$ in $L^1$}
Theorem \ref{main 2} can be employed to show the compactness of the ellipsoidal Gaussian by enabling one to derive the Diperna-Lions type inequality \cite{D-L1} in the range $-1/2<\nu<1$, when $A_{\nu}$ does not depend on macroscopic fields. We first consider
\begin{align*}
\mathcal{M}_{\nu}(f)-f&=\big\{\mathcal{M}_{\nu}(f)-f\big\}\big(1_{\mathcal{M}_{\nu}(f)<Rf^n}+1_{\mathcal{M}_{\nu}(f)>Rf}\big)\cr
&\leq(R-1)f1_{\mathcal{M}_{\nu}(f)<Rf}\cr
&+\frac{1}{\ln R}\big(\mathcal{M}_{\nu}(f)-f\big)\big(\ln\mathcal{M}_{\nu}(f)-\ln f\big)1_{\mathcal{M}_{\nu}(f)\geq Rf},
\end{align*}
where $R>1$.
This gives, for any fixed positive number $T^f$ and measurable set $B_{x,v}\subset\Omega_x\times\mathbb{R}^3_v$,
\begin{eqnarray*}
&&\int^{T^f}_0\int_{B_{x,v}}\mathcal{M}_{\nu}(f)dxdvdt\cr
&&\quad\leq R\int^{T^f}_0\int_{B_{x,v}}fdxdvdt+\frac{1}{\ln R}\int^{T^f}_0\int_{ B_{x,v}}\big(\mathcal{M}_{\nu}(f)-f\big)\big(\ln\mathcal{M}_{\nu}(f)-\ln f\big)dvdxdt\cr
&&\quad\leq R\int^{T^f}_0\int_{B_{x,v}}fdxdvdt+\frac{1}{\ln R}\int^{T^f}_0\int_{ \Omega_x\times \mathbb{R}^3_v}\big(\mathcal{M}_{\nu}(f)-f\big)\big(\ln\mathcal{M}_{\nu}(f)-\ln f\big)dvdxdt\cr
&&\quad= R\int^{T^f}_0\int_{B_{x,v}}fdxdvdt
+\frac{1}{\ln R}\int^{T^f}_0\int_{ \Omega_x}D_{\nu}(f)dxdt-\frac{1}{\ln R}\int^{T^f}_0\int_{\Omega_x}R_{\nu}(f)dxdt\cr
&&\quad\leq R\int^{T^f}_0\int_{B_{x,v}} fdxdvdt+\frac{1}{\ln R}\left\{\int_{\Omega_x\times\mathbb{R}^3_v}f_0|\ln f_0|dxdv+C_{f_0,T^f}\right\}\cr
&&\quad-\frac{1}{\ln R}\int^{T^f}_0\int_{\Omega_x}R_{\nu}(f)dxdt.
\end{eqnarray*}
Here we used
\[
\int_{\Omega_x\times \mathbb{R}^3_v}f(t)|\ln f(t)|dxdv+\int^t_0\int_{\Omega_x}D_{\nu}(f(s))dxds\leq
\int_{\Omega_x\times \mathbb{R}^3_v}f_0|\ln f_0|dxdv+C_{f_0,t},
\]
which follows by a standard argument from (\ref{corol1})\cite{D-L1}.
When $0<\nu<1$, the last term is non-negative, and can be ignored:
\begin{align*}
\int^{T^f}_0\int_{B_{x,v}}\mathcal{M}_{\nu}(f)dxdvdt&\leq
R\int^{T^f}_0\int_{B_{x,v}} fdxdvdt\cr&+\frac{1}{\ln R}\left\{\int_{\Omega_x\times \mathbb{R}^3_v}f_0|\ln f_0|dxdv+C_{f_0,T^f}\right\}.
\end{align*}
In the case $-1/2<\nu\leq 0$, we recall
\[
R_{\nu}=A_{\nu}\rho(3-F_{\nu}),
\]
and
\[
(1-\nu)T+\nu\theta_i=(1+2\nu)T-\nu\sum_{j\neq i}\theta_j,
\]
to compute
\begin{eqnarray*}
&&-\int^{T^f}_0\int_{\Omega_x}R_{\nu}(x,t)dxdt\cr
&&\hspace{0.9cm}=\int^{T^f}_0\int_{\Omega_x}A_{\nu}\rho\left\{-3+\sum_{1\leq i\leq 3}\frac{\theta_i}{(1-\nu)T+\nu\theta_i}\right\}dxdt\cr
&&\hspace{0.9cm}=-3A_{\nu}\int^{T^f}_0\int_{\Omega_x\times \mathbb{R}^3_v}fdxdvdt+A_{\nu}\sum_{1\leq i\leq 3}\int^{T^f}_0\int_{ \Omega_x}\frac{\theta_i\rho}{(1-\nu)T+\nu\theta_i} dxdt\cr
&&\hspace{0.9cm}=-3A_{\nu}\int^{T^f}_0\int_{\Omega_x\times \mathbb{R}^3_v}fdxdvdt+A_{\nu}\sum_{1\leq i\leq 3}\int^{T^f}_0\int_{\Omega_x}\frac{\theta_i\rho}{(1+2\nu)T-\nu\sum_{j\neq i}\theta_j} dxdt\cr
&&\hspace{0.9cm}\leq-3A_{\nu}\int^{T^f}_0\int_{\Omega_x\times \mathbb{R}^3_v}fdxdvdt+A_{\nu}\sum_{1\leq i\leq 3}\int^{T^f}_0\int_{\Omega_x}\frac{\theta_i\rho}{(1+2\nu)T} dxdt\cr
&&\hspace{0.9cm}\leq-3A_{\nu}\int^{T^f}_0\int_{\Omega_x\times \mathbb{R}^3_v}fdxdvdt+\frac{3A_{\nu}}{1+2\nu}\int^{T^f}_0\int_{\Omega_x}\rho dxdt\cr
&&\hspace{0.9cm}\leq\frac{-6\nu}{(1-\nu)(1+2\nu)}\int^{T^f}_0\int_{\Omega_x\times \mathbb{R}^3_v}fdxdvdt\cr
&&\hspace{0.9cm}=\frac{-6\nu T_f}{(1-\nu)(1+2\nu)}\int_{\Omega_x\times \mathbb{R}^3_v}f_0dxdv.
\end{eqnarray*}
In summary, we have
\begin{align*}
\int^{T^f}_0\int_{B_{x,v}}\mathcal{M}_{\nu}(f)dxdvdt&\leq R\int^{T^f}_0\int_{B_{x,v}} fdxdvdt\cr
&+\frac{1}{\ln R}\left\{\int_{\Omega_x\times\mathbb{R}^3_v}f_0|\ln f_0|dxdv+C_{f_0,T^f}\right\}\cr
&+\frac{1}{\ln R}\left\{\frac{-6\nu T_f}{(1-\nu)(1+2\nu)}\int_{\Omega_x\times \mathbb{R}^3_v}f_0dxdv\right\}1_{-1/2<\nu<1}.
\end{align*}
which, from the Dunford-Pettis theroem, gives the weak compactness of $\mathcal{M}_{\nu}(f)$ in $L^1((0,T)\times \Omega_x\times \mathbb{R}^3_v)$ once
$\{f\}$ is weak compact in $L^1((0,T)\times \Omega_x\times \mathbb{R}^3_v)$.
\begin{section}{Acknowledgement}
This research was supported by Basic Science Research Program through
the National Research Foundation of Korea (NRF) funded by the Ministry of Science, ICT $\&$
Future Planning (NRF-2014R1A1A1006432)
\end{section}


%
%
\bibliographystyle{amsplain}

\end{document}